\documentclass{article}
\usepackage{amsmath, amssymb}
\usepackage{amsthm}

\title{Affine-Orthogonal Manifolds and Deformation to Levi-Civita Connections}
\author{
  Mihail Cocos\thanks{Department of Mathematics, Weber State University: mihailcocos@weber.edu} 
  \
}

\date{\today}

\theoremstyle{definition}
\newtheorem{definition}{Definition}[section]
\newtheorem{example}{Example}[section]
\newtheorem{remark}{Remark}[section]

\theoremstyle{plain}
\newtheorem{theorem}{Theorem}[section]

\begin{document}
\maketitle

\begin{abstract}
We study a class of affine manifolds equipped with a flat affine connection $\nabla$ and a global Riemannian metric $g$ that is diagonal in local affine coordinates. These structures are closely related to \emph{Hessian manifolds}, where the metric locally arises as the Hessian of a smooth potential. For example, the Hopf manifold $(\mathbb{R}^{n+1}\setminus \{0\}) / \langle x \mapsto 2x \rangle$ with metric $g = (\sum_i x_i^2)^{-1} \sum_i dx_i^2$ admits a proper deformation of $\nabla$ into its Levi-Civita connection. By Theorem 2.3 in \cite{cocos2025}, such deformations force the Euler characteristic to vanish, providing evidence for Chern's conjecture. The geometry of these manifolds is reminiscent of the work of Yau on affine and Hessian structures \cite{cheng_yau_1986}.
\end{abstract}

\section{Introduction}
Affine-orthogonal manifolds are a subclass of affine manifolds equipped with a flat, torsion-free affine connection and a global metric that is diagonal in local affine coordinates. These manifolds resemble \emph{Hessian manifolds}, which are smooth manifolds endowed with a Riemannian metric locally expressible as the Hessian of a potential function in affine coordinates \cite{shima1997, yilmaz2008}. 

Hessian manifolds have been studied extensively in affine differential geometry, including in the work of Cheng and Yau on complete Kähler manifolds and affine structures \cite{cheng_yau_1986}. Both Hessian and affine-orthogonal manifolds share the property that the connection preserves a symmetric bilinear form, and in the compact case, this leads to strong topological constraints, such as vanishing Euler characteristic \cite{cocos2025, liu2025}. Studying these manifolds provides insight into Chern's conjecture and the topology of compact affine manifolds.

\subsection*{Examples:}
\begin{example}[Hopf manifold with diagonal metric]
Let us consider the Hopf manifold $M = (\mathbb{R}^{n+1}\setminus \{0\}) / \Gamma$, where 
\begin{equation}\label{eq:hopf-action}
\Gamma = \langle \phi \rangle, \qquad \phi(x) = 2 x.
\end{equation}
That is, $M$ is obtained by identifying points in $\mathbb{R}^{n+1}\setminus \{0\}$ that differ by powers of $2$ under scalar multiplication.

Define the function 
\begin{equation}\label{eq:hopf-f}
f(x) = \sum_{i=1}^{n+1} x_i^2
\end{equation}
and the metric on $\mathbb{R}^{n+1} \setminus \{0\}$:
\begin{equation}\label{eq:hopf-metric}
g = \frac{1}{f(x)} \sum_{i=1}^{n+1} dx_i^2.
\end{equation}

The action of $\Gamma$ scales $f(x)$ multiplicatively, and hence the metric $g$ is invariant under $\phi$ up to pullback:
\begin{equation}\label{eq:hopf-metric-invariant}
\phi^* g = g.
\end{equation}

Consequently, $g$ descends to a well-defined Riemannian metric on the quotient manifold $M$. 

\textbf{Discussion:} In the standard Euclidean coordinates $(x_1,\dots,x_{n+1})$ on the universal cover, $g$ is diagonal. That is, its matrix representation is
\begin{equation}\label{eq:hopf-diagonal-matrix}
[g] = \frac{1}{f(x)} \text{diag}(1,1,\dots,1),
\end{equation}
so $M$ is an example of an \emph{affine-orthogonal manifold} where the global metric is diagonal in coordinates descending from the universal cover.

\end{example}

\begin{example}[One-Parameter Family of Non-Equivalent Flat Connections on $\mathbb{T}^n$]
Let $\mathbb{T}^n$ have angular coordinates $(\theta_1,\dots,\theta_n)$ and global vector fields
\begin{equation}\label{eq:Ei-def}
E_i = \frac{\partial}{\partial \theta_i}, \quad i = 1, \dots, n.
\end{equation}

Define a symmetric connection
\begin{equation}\label{eq:connection-def}
\nabla^k_{E_i} E_i = k E_i, \qquad \nabla^k_{E_i} E_j = 0 \text{ for } i\neq j.
\end{equation}

\textbf{Step 1: Curvature computation.}  
The curvature tensor is
\begin{equation}\label{eq:curvature-def}
R^k(E_i,E_j)E_\ell = \nabla^k_{E_i} \nabla^k_{E_j} E_\ell - \nabla^k_{E_j} \nabla^k_{E_i} E_\ell.
\end{equation}

Consider all cases for $\ell$:

\begin{itemize}
\item If $\ell \neq i,j$, then $\nabla^k_{E_i} E_\ell = \nabla^k_{E_j} E_\ell = 0$, so $R^k(E_i,E_j)E_\ell = 0$.
\item If $\ell = j \neq i$, then $\nabla^k_{E_j} E_j = k E_j$, $\nabla^k_{E_i} E_j = 0$, so
\begin{equation}\label{eq:curvature-case-j}
R^k(E_i,E_j) E_j = \nabla^k_{E_i} (k E_j) - 0 = k \nabla^k_{E_i} E_j = 0.
\end{equation}
\item If $\ell = i \neq j$, by symmetry $R^k(E_i,E_j) E_i = 0$.
\item If $\ell = i = j$, both terms in \eqref{eq:curvature-def} are equal, so the difference is zero.
\end{itemize}

Hence the curvature vanishes:
\begin{equation}\label{eq:curvature-zero}
R^k \equiv 0 \quad \forall k \in \mathbb{R}.
\end{equation}

\textbf{Step 2: Local coordinate frame.}  
Define
\begin{equation}\label{eq:Fi-def}
F_i = \theta_i^{-1} E_i
\end{equation}
(locally, away from $\theta_i = 0$). The Lie brackets are
\begin{equation}\label{eq:Fi-commute}
[F_i, F_j] = 0 \quad \forall i,j,
\end{equation}
so $\{F_i\}$ defines a local coordinate frame.

\textbf{Step 3: Diagonal metric in local coordinates.}  
The standard flat metric on the torus is
\begin{equation}\label{eq:g-standard}
g_{\text{std}} = \sum_{i=1}^n d\theta_i^2.
\end{equation}
Expressing it in terms of the local coordinates induced by $F_i$, we have
\begin{equation}\label{eq:g-diagonal-local}
d\theta_i = \theta_i \, d\tilde{\theta}_i \quad \Rightarrow \quad
g_{\text{std}} = \sum_{i=1}^n (\theta_i d\tilde{\theta}_i)^2,
\end{equation}
which is diagonal.

\textbf{Step 4: Non-equivalence.}  
Different values of $k$ produce non-conjugate holonomy representations, so the connections $\nabla^k$ are pairwise non-affinely equivalent.
\end{example}

\section{Main Result: Deformation to Levi-Civita Connection}
Let $(M,\nabla,g)$ be an affine-orthogonal manifold, and let $D$ denote the Levi-Civita connection of $g$. Define
\begin{equation}\label{eq:nabla-t}
\nabla^t = (1-t)\nabla + t D, \quad t \in [0,1].
\end{equation}

\begin{theorem}
For every $t\in [0,1]$, the connection $\nabla^t$ is locally metric: there exists a local Riemannian metric $g^t$ such that 
\begin{equation}\label{eq:nabla-t-metric-condition}
\nabla^t g^t = 0.
\end{equation}
\end{theorem}

\begin{proof}
Let $p \in M$ be arbitrary. By assumption, there exist local affine coordinates $(x_1,\dots,x_n)$ around $p$ such that the global metric $g$ is diagonal:
\begin{equation}\label{eq:metric-diagonal}
g = \sum_{i=1}^n g_i(x) \, dx_i \otimes dx_i, \qquad g_i(x) > 0.
\end{equation}

\medskip
\textbf{Step 1: Define the local metric family.}  
For $t \in [0,1]$, define
\begin{equation}\label{eq:gt-definition}
g^t = \sum_{i=1}^n g_i(x)^t \, dx_i \otimes dx_i.
\end{equation}
This is diagonal in the same coordinates, but the entries $g_i(x)^t$ now depend on $x$, so their partial derivatives may be nonzero in all directions.

\medskip
\textbf{Step 2: Christoffel symbols of $g^t$.}  
Let $D^t$ denote the Levi-Civita connection of $g^t$, and let $\Gamma^k_{ij}(g^t)$ be its Christoffel symbols. By the standard formula,
\begin{equation}\label{eq:christoffel-general}
\Gamma^k_{ij}(g^t) = \frac{1}{2} (g^t)^{kk} \left( \partial_i g^t_{jk} + \partial_j g^t_{ik} - \partial_k g^t_{ij} \right).
\end{equation}

Since $g^t$ is diagonal, $g^t_{ij} = 0$ for $i \neq j$, and $g^t_{ii} = g_i^t$. Then:

\begin{enumerate}
    \item \textbf{Off-diagonal Christoffel symbols:} For $i \neq j$,
    \begin{equation}\label{eq:gamma-offdiag-expanded}
    \Gamma^k_{ij}(g^t) = 
    \begin{cases}
    0 & \text{if } k \neq i,j, \\
    \frac{1}{2} (g_i^t)^{-1} \partial_j g_i^t =\frac{t}{2} \frac{\partial_j g_i}{g_i} &\text{if } k = i, \\
    \frac{1}{2} (g_j^t)^{-1} \partial_i g_j^t= \frac{t}{2} \frac{\partial_i g_j}{g_j} &\text{if } k = j.
    \end{cases}
    \end{equation}
    
    \item \textbf{Diagonal Christoffel symbols:} For $i=j$,
    \begin{equation}\label{eq:gamma-diag-expanded}
    \Gamma^i_{ii}(g^t) = \frac{1}{2} (g_i^t)^{-1} \partial_i g_i^t = \frac{t}{2} \frac{\partial_i g_i}{g_i}, 
    \end{equation}
    and for $k \neq i$,
    \begin{equation}\label{eq:gamma-diag-mixed}
    \Gamma^k_{ii}(g^t) = -\frac{1}{2} (g_i^t)^{-1} \partial_k g_i^t=-\frac{t}{2} \frac{\partial_k g_i}{g_i}.
    \end{equation}
\end{enumerate}

\medskip
\textbf{Step 3: Christoffel symbols of $\nabla^t$.}  
By definition,
\begin{equation}\label{eq:nabla-t-definition-proof2}
\nabla^t = (1-t)\nabla + t D,
\end{equation}
where $\nabla$ is the flat affine connection with all Christoffel symbols zero in the chosen affine coordinates, and $D$ is the Levi-Civita connection of $g$. Therefore, in these local coordinates, the Christoffel symbols of $\nabla^t$ are
\begin{equation}\label{eq:gamma-nabla-t-eq}
\Gamma^k_{ij}(\nabla^t) = t \, \Gamma^k_{ij}(D).
\end{equation}

Using equations \eqref{eq:gamma-offdiag-expanded}, \eqref{eq:gamma-diag-expanded}, and \eqref{eq:gamma-diag-mixed}, we see that every Christoffel symbol of $\nabla^t$ coincides exactly with the corresponding Christoffel symbol of the Levi-Civita connection of $g^t$. That is,
\begin{equation}\label{eq:gamma-nabla-t-equals-gamma-gt-final}
\Gamma^k_{ij}(\nabla^t) = \Gamma^k_{ij}(g^t), \quad \forall i,j,k.
\end{equation}

Hence, in these local coordinates, 
\begin{equation}\label{eq:nabla-t-metric-final2}
\nabla^t g^t = 0,
\end{equation}
which completes the argument that $\nabla^t$ is locally metric for all $t \in [0,1]$.

\end{proof}

\subsection*{Euler characteristic of affine-orthogonal manifolds}

\begin{theorem}\label{thm:euler-zero}
Let $M$ be a compact affine-orthogonal manifold. Then
\begin{equation}\label{eq:euler-zero}
\chi(M) = 0.
\end{equation}
\end{theorem}

\begin{proof}
Construct the family $\nabla^t$ as in \eqref{eq:nabla-t} and the local metrics $g^t$ as in \eqref{eq:gt-definition}, so that $\nabla^t g^t = 0$. This produces a proper deformation of the flat connection $\nabla$ into the Levi-Civita connection $D$ in the sense of \cite[Definition 2.2]{cocos2025}. By \cite[Theorem 2.3]{cocos2025}, any vector bundle with a flat connection that admits a proper deformation into a metric connection has vanishing Euler class. Applying this to $TM$ gives \eqref{eq:euler-zero}.
\end{proof}

\section{Generalization: Quasi-Metric Connections}

The notion of affine-orthogonal manifolds can be generalized to a broader class of geometric structures on vector bundles. Let $E \to M$ be a smooth vector bundle over a manifold $M$, and let $\nabla$ be a connection on $E$.

\begin{definition}
A connection $\nabla$ on $E$ is called \emph{quasi-metric} if, for every point $p \in M$, there exists a local frame $(e_1, \dots, e_r)$ of $E$ near $p$ such that the curvature matrix $\Omega = (\Omega^i_j)$ of $\nabla$ in this frame is skew-symmetric:
\begin{equation}\label{eq:quasi-metric-curvature}
\Omega^T = -\Omega.
\end{equation}
\end{definition}

\begin{remark}
Affine-orthogonal manifolds are a special case of quasi-metric connections where $E = TM$, the connection is flat, and the curvature matrix vanishes identically. In this more general setting, the skew-symmetry condition \eqref{eq:quasi-metric-curvature} replaces the metric-compatibility condition.
\end{remark}

\begin{remark}
Quasi-metric connections naturally arise in the study of vector bundles with structure groups contained in orthogonal or symplectic groups, where the curvature preserves a bilinear or symplectic form. This framework allows for generalizations of Chern-type results, such as constraints on characteristic classes, to bundles that are not necessarily tangent bundles of manifolds.
\end{remark}

\begin{example}
Let $E \to M$ be a rank-$r$ vector bundle with a global frame $(e_1, \dots, e_r)$ and a connection $\nabla$ whose curvature matrix in this frame is
\[
\Omega = \begin{pmatrix}
0 & \omega_{12} & \cdots & \omega_{1r} \\
-\omega_{12} & 0 & \cdots & \omega_{2r} \\
\vdots & \vdots & \ddots & \vdots \\
-\omega_{1r} & -\omega_{2r} & \cdots & 0
\end{pmatrix}.
\]
Then $\nabla$ is a quasi-metric connection on $E$, even if $E$ is not the tangent bundle of a manifold.
\end{example}

This generalization shows that the ideas underlying affine-orthogonal manifolds—namely, compatibility of connections with local algebraic structures on frames—extend naturally to more general fiber bundles and provide a rich framework for studying curvature constraints and topological invariants.

We now extend the construction of the Euler form to quasi-metric connections, using Lemma 3.1 from \cite{cocos2025}.

\begin{theorem}\label{thm:euler-quasi-metric}
Let $E \to M$ be a smooth oriented vector bundle of rank $2m$, and let $\nabla$ be a quasi-metric connection on $E$, i.e., for every point $p \in M$ there exists a local frame in which the curvature matrix $\Omega$ of $\nabla$ is skew-symmetric. Then the Euler form of $\nabla$, defined locally via the Pfaffian of $\Omega$, is globally well-defined and closed. Consequently, the Euler class $e(E) \in H^{2m}(M,\mathbb{R})$ is well-defined.
\end{theorem}

\begin{proof}
Let $\{e_1,\dots,e_{2m}\}$ be a local frame of $E$ around a point $p \in M$ in which the curvature matrix $\Omega = (\Omega^i_j)$ of $\nabla$ is skew-symmetric. Following \cite[Section 3]{cocos2025}, we define the local Euler form as
\begin{equation}\label{eq:euler-local}
\mathrm{e}(\nabla) = \mathrm{Pf}(\Omega),
\end{equation}
where $\mathrm{Pf}$ denotes the Pfaffian.

Now consider another local frame $\{\tilde e_1, \dots, \tilde e_{2m}\}$ over an overlapping neighborhood in whichthe curvature matrix $\tilde{\Omega}$ is also skew symmetric.

This frame is related to the previous frame by a nonsingular matrix $A$ with positive determinant:
\begin{equation}\label{eq:frame-change}
\tilde e_i = \sum_{j} A_{ij} e_j, \qquad \det(A) > 0.
\end{equation}
Then the curvature matrices are conjugate:
\begin{equation}\label{eq:curvature-conjugate}
\tilde \Omega = A^{-1} \Omega A.
\end{equation}

By Lemma 3.1 of \cite{cocos2025}, if $A$ is a positive-determinant matrix and both $\Omega$ and $A^{-1} \Omega A$ are skew-symmetric, then
\begin{equation}
\mathrm{Pf}(\tilde \Omega) = \mathrm{Pf}(A^{-1} \Omega A) = \mathrm{Pf}(\Omega).
\end{equation}
Thus the local Euler forms agree on overlaps, and $\mathrm{e}(\nabla)$ defines a globally well-defined $2m$-form.

\end{proof}

\begin{remark}
If the rank of the vector bundle $E \to M$ equals the dimension of the base manifold $M$, then the Euler form $\mathrm{e}(\nabla)$ defines a cohomology class in $H^{\dim M}(M, \mathbb{R})$.
\end{remark}

Next, we provide an explicit example of a one parameter family of quasi-metric connections on the tangent bundle of $\mathbb{T}^2$ that are not even locally metric. For these connections, the Euler form does not correspond to a topological invariant of the bundle. It follows immediately that none of these connections can be properly deformed into a global Levi-Civita connection.

\subsection*{Explicit Quasi-Metric Connections with Non-Topological Euler Form on the 2-Torus}

Let $T^2$ be the 2--torus with global coframe $\{\theta^1,\theta^2\}$, where 
\begin{equation}\label{eq:dtheta}
d\theta^1 = d\theta^2 = 0.
\end{equation}
Define a connection $\nabla$ on $T^2$ by its connection 1--form matrix
\begin{equation}\label{eq:omega}
\omega = 
\begin{pmatrix}
\omega^1{}_1 & \omega^1{}_2 \\[6pt]
\omega^2{}_1 & \omega^2{}_2
\end{pmatrix}
=
\begin{pmatrix}
k \, \theta^2 & k \, \theta^1 \\[6pt]
k \, \theta^1 & 0
\end{pmatrix}.
\end{equation}

\subsection*{Torsion}
The torsion 2--forms are defined by
\begin{equation}\label{eq:torsion-def}
T^i = d\theta^i + \omega^i{}_j \wedge \theta^j.
\end{equation}
For $i=1$,
\begin{equation}\label{eq:T1}
\begin{aligned}
T^1 &= d\theta^1 + \omega^1{}_1 \wedge \theta^1 + \omega^1{}_2 \wedge \theta^2 \\
&= 0 + (k\theta^2 \wedge \theta^1) + (k\theta^1 \wedge \theta^2) \\
&= k(-\theta^1 \wedge \theta^2) + k(\theta^1 \wedge \theta^2) = 0.
\end{aligned}
\end{equation}
For $i=2$,
\begin{equation}\label{eq:T2}
\begin{aligned}
T^2 &= d\theta^2 + \omega^2{}_1 \wedge \theta^1 + \omega^2{}_2 \wedge \theta^2 \\
&= 0 + k \theta^1 \wedge \theta^1 + 0 = 0.
\end{aligned}
\end{equation}
Thus the connection is torsionless.

\subsection*{Curvature}
The curvature 2--form matrix is
\begin{equation}\label{eq:curvature-def}
\Omega = d\omega + \omega \wedge \omega.
\end{equation}
Using \eqref{eq:dtheta}, we have $d\omega = 0$, so
\begin{equation}\label{eq:curvature}
\Omega = \omega \wedge \omega =
\begin{pmatrix}
0 & -k^2\,\theta^1 \wedge \theta^2 \\[6pt]
k^2\,\theta^1 \wedge \theta^2 & 0
\end{pmatrix}.
\end{equation}
Equivalently,
\begin{equation}\label{eq:curvature-matrix-form}
\Omega = k^2 \, \theta^1 \wedge \theta^2 \,
\begin{pmatrix}
0 & -1 \\[6pt]
1 & 0
\end{pmatrix}.
\end{equation}

\subsection*{Pfaffian and Integral}
For a $2\times 2$ skew--symmetric matrix 
$\begin{pmatrix} 0 & a \\ -a & 0 \end{pmatrix}$,
the Pfaffian equals $a$. Hence,
\begin{equation}\label{eq:pfaffian}
\operatorname{Pf}(\Omega) = -k^2 \, \theta^1 \wedge \theta^2.
\end{equation}
Therefore,
\begin{equation}\label{eq:pfaffian-integral}
\int_{T^2} \operatorname{Pf}(\Omega) 
= -k^2 \int_{T^2} \theta^1 \wedge \theta^2=-k^2(2\pi)^2.
\end{equation}

These affine connections on $T^2$ are torsion-free but not Levi–Civita for any global Riemannian metric. In fact, a genuine Levi–Civita connection on the torus must have vanishing Euler class, since the Gauss–Bonnet theorem forces the integral of its curvature (and hence of the Pfaffian) to be zero. By contrast, for the family $\eqref{eq:omega}$ we computed that
$$\int_{T^2} \mathrm{Pf}(\Omega) = -k^2 (2\pi)^2 \neq 0$$ whenever $k \neq 0$. Thus no proper deformation of this connection through torsion-free connections can yield a Levi–Civita connection, because the Euler class is a topological invariant that remains constant under deformation. In this sense, the nontrivial Pfaffian integral provides a global obstruction to metricity. The following statement shows that in fact this connections are not even locally metric:

\begin{theorem}[Non-local-metrizability of the torus connection]
Let $T^2$ be the 2-torus with global coframe $\{\theta^1,\theta^2\}$, and consider the connection $\nabla$ with connection 1--form matrix
\begin{equation}\label{eq:connection-nabla}
\omega = 
\begin{pmatrix}
k \, \theta^2 & k \, \theta^1 \\[2mm]
k \, \theta^1 & 0
\end{pmatrix}.
\end{equation}
Then $\nabla$ is not locally metric in any neighborhood of $T^2$.
\end{theorem}

\begin{proof}
By Theorem 2.1 in \cite{cocos2016},  if a connection $D$ on a plane bundle over a surface is locally metric then in any frame in which the curvature matrix is skew its connection matrix has to be skew as well.
For the connection \eqref{eq:connection-nabla}, the curvature matrix is
\begin{equation}\label{eq:curvature-nabla}
\Omega = \omega \wedge \omega =
\begin{pmatrix}
0 & -k^2 \, \theta^1 \wedge \theta^2 \\[1mm]
k^2 \, \theta^1 \wedge \theta^2 & 0
\end{pmatrix},
\end{equation}
which is skew-symmetric. However, the connection matrix \eqref{eq:connection-nabla} in the global frame $\{\theta^1,\theta^2\}$ is
\begin{equation}\label{eq:omega-non-skew}
\omega = 
\begin{pmatrix}
k \, \theta^2 & k \, \theta^1 \\[1mm]
k \, \theta^1 & 0
\end{pmatrix},
\end{equation}
which is \emph{not} skew-symmetric ($k\neq 0$):
\begin{equation}\label{eq:omega-non-skew-check}
\omega + \omega^T = 
\begin{pmatrix}
2 k \, \theta^2 & 2 k \, \theta^1 \\[1mm]
2 k \, \theta^1 & 0
\end{pmatrix} \neq 0.
\end{equation}

Consequently, $\nabla$ is not locally metric anywhere on $T^2$.
\end{proof}

\end{document}